\documentclass[12pt]{amsart}
\usepackage{tikz}
\usepackage{amsmath}
\usepackage{amsxtra}
\usepackage{amscd}
\usepackage{amsthm}
\usepackage{amsfonts}
\usepackage{amssymb}
\usepackage {pstricks}
\usepackage{pstricks,pst-node}
\usepackage[all]{xy}

\newtheorem{theorem}{Theorem}[section]

\newtheorem{lemma}[theorem]{Lemma}

\newtheorem{corollary}[theorem]{Corollary}
\newtheorem{proposition}[theorem]{Proposition}

\theoremstyle{definition}
\newtheorem{definition}[theorem]{Definition}
\newtheorem{example}[theorem]{Example}

\theoremstyle{remark}
\newtheorem{remark}[theorem]{Remark}

\numberwithin{equation}{section}

\newcommand{\mc}{\mathcal}

\newcommand{\C}{{\mathbb C}}
\newcommand{\R}{{\mathbb R}}
\newcommand{\Z}{{\mathbb Z}}
\newcommand{\N}{{\mathbb N}}
\newcommand{\F}{{\mathbb F}}

\newcommand{\CT}{{\mathcal T}}

\newcommand{\Rad}{{\rm{Rad}}}

\newcommand{\id}{{\rm{id}}}
\newcommand{\cC}{\mc C}

\newcommand{\cL}{\mc L}
\newcommand{\be}{\begin{equation}}
\newcommand{\ee}{\end{equation}}
\newcommand{\cR}{\mc R}

\newcommand{\bT}{\mathbf T}
\newcommand{\bw}{\mathbf w}
\newcommand{\bl}{\mathbf L}

\newcommand{\inv}{^{-1}}
\newcommand{\ol}{\overline}

\newcommand{\lr}{\longrightarrow}
\newcommand{\wt}{\widetilde}

\newcommand{\U}{{\rm{U}}}

\newcommand{\Hom}{{\rm{Hom}}}

\newcommand{\TL}{{\rm{TL}}}

\newcommand{\tr}{{\rm{tr}}}

\advance\headheight by 2pt

\newcommand{\fso}{{\mathfrak {so}}}

\newcommand{\ot}{\otimes}

\newcommand{\SO}{{\rm SO}}

\newcommand{\la}{{\langle}}
\newcommand{\ra}{{\rangle}}

\begin{document}

\normalfont

\title[Temperley-Lieb and Virosoro]{The Jones quotients of the Temperley-Lieb algebras.}

%{\huge\color{red}FINAL??}

\author{K. Iohara, G.I. Lehrer and R.B. Zhang}
\thanks{The present work was initiated during an Australian Research Council funded
visit of K.~I. to the University of Sydney in October-November 2016. He 
gratefully acknowledges the support and hospitality extended to him.}
\address{Univ Lyon, Universit\'{e} Claude Bernard Lyon 1, CNRS UMR 5208, Institut Camille Jordan, 
43 Boulevard du 11 Novembre 1918, F-69622 Villeurbanne cedex, France}
\email{iohara@math.univ-lyon1.fr}
\address{School of Mathematics and Statistics,
University of Sydney, N.S.W. 2006, Australia}
\email{gustav.lehrer@sydney.edu.au, ruibin.zhang@sydney.edu.au}
\begin{abstract} When the parameter $q$ is a root of unity, the Temperley-Lieb algebra $\TL_n(q)$ is non-semisimple
for almost all $n$. Jones showed that there is a canonical symmetric bilinear form on $\TL_n(q)$, whose radical $R_n(q)$ is generated by
a certain idempotent $E_\ell\in\TL_{\ell-1}(q)\subseteq\TL_n(q)$, which is now referred to as 
the Jones-Wenzl idempotent, for which an explicit formula was subsequently given by Graham and Lehrer.
In this work, we study the quotients $Q_n(\ell):=\TL_n(q)/R_n(q)$, where $|q^2|=\ell$, which are precisely the algebras generated by
Jones' projections. We give the dimensions of their simple modules,
as well as $\dim(Q_n(\ell))$; en route we give generating functions and recursions for the dimensions of cell modules and associated combinatorics.
When the order $|q^2|=4$, we obtain an isomorphism of $Q_n(\ell)$ with the even part of the Clifford algebra, well known to physicists
through the Ising model. When $|q^2|=5$, we obtain a sequence of algebras whose dimensions are the odd-indexed Fibonacci numbers.
The general case is described explicitly.
\end{abstract}
%\date {31 March  2013}
%\date{\today}
\maketitle
%%%%%%%%%%%%%%%%%%%%%%%%%%%%%%%%%%%%%%%%%%%%%%%%%%%
%%%%%%%%%%%%%%%%%%%%%%%%%%%%%%%%%%%%%%%%%%%%%%%%%%%
\section{Introduction}
\subsection{Preamble} In a celebrated series of works, Jones \cite{J2,J1,J3} showed that a certain sequence of von Neumann 
algebras $A_n$, $n=1,2,3,\dots$ generated by idempotents $\{e_1,\dots,e_{n-1}\}$ in a type $\text{II}_1$ factor
is closely related to the Temperley-Lieb algebras. His algebras depend on a parameter, denoted $q^2\in\C$ in the present work,
whose value in Jones' work was either in $\R_{\geq 0}$, or equal to $\exp(\pm \frac{2\pi i}{\ell})$ for $\ell=3,4,5,\dots$. 
This constraint arose because of the context of the $e_i$ as projections in a type $\text{II}_1$ factor (see \cite{W}),
which implies that a certain trace form on the algebra is positive definite.
When $q^2\in\R$, the projections $e_i$ 
generate precisely the Temperley-Lieb algebra (see \ref{def:tln} below), and the positive definite nature of
the trace form is automatic. 

However the requirement that Jones' trace on $A_n$ be non-degenerate
implies that for $q^2=\exp(\pm \frac{2\pi i}{\ell})$, the relevant algebra $A_n$ is a proper quotient of
the Temperley-Lieb algebra $\TL_n(q)$ (Definition \ref{def:tln} below) by 
the radical of the Jones form.

This radical of the trace form, regarded as a form on $\TL_n(q)$,
is known \cite{J1,GLA}, \cite[Thm. 2.1]{JR} to be generated by a single idempotent in $\TL_{\ell-1}(\exp(\pm \frac{2\pi i}{\ell}))$, 
known as the Jones-Wenzl idempotent. The main purpose of this work is to completely determine the dimension and representations 
of the quotient of $\TL_{\ell-1}(\exp(\pm \frac{2\pi i}{\ell}))$ by this radical, 
which we denote in this work by $Q_n(\ell)$ and which we refer to as the ``Jones quotient''. 
When $\ell=4$ we show below (Theorem \ref{thm:qcl})
 that, as is well known in the physics
literature \cite{KS,MLW}, $Q_n(4)$ is isomorphic to the even part of the standard Clifford algebra.  
Our work could therefore be interpreted as the study
of certain generalisations of the Clifford algebras. Other such generalisations exist in the literature (cf. \cite{MLW,M1}), 
but are quite different from the 
Jones projection algebras $Q_n(\ell)$.

Our interest in this subject stemmed originally from a desire to understand how the Virasoro algebra arises as a limit of the Temperley-Lieb
algebras, or of their semisimple quotients, which is a recurring theme in conformal field theory (see, for example, \cite{KS} 
and \S\ref{ss:spec} below). We hope 
to return to this in the future. 

\subsection{The Temperley-Lieb algebras}
In this work, all algebras will be over $\C$. Much of the theory we develop applies over more general 
domains, but since we will be concerned here with connections to the theory of
operator algebras and mathematical physics, we limit our discussion to $\C$-algebras. For $n\in\N$, the Temperley-Lieb algebra
$\TL_n(q)$ is defined as follows.

\begin{definition}\label{def:tln}
Let $q\in\C$. $\TL_n=\TL_n(q)$ is the associative $\C$-algebra with generators $f_1,f_2,\dots,f_{n-1}$ and relations
\be\label{eq:reltl}
\begin{aligned}
f_i^2=&-(q+q\inv)f_i\text{ for all }i\\
f_if_{i\pm 1}f_i=&f_i\text{ for all }i\\
f_if_j=&f_jf_i\text{ if }|i-j|\geq 2.\\
\end{aligned}
 \ee 
\end{definition}

In his seminal work \cite{J2} on subfactors of a factor, Jones showed that certain projectors 
$\{e_1,\dots,e_{n-1}\}$ ($n=1,2,,3,\dots$) in a von Neumann algebra satisfy the Temperley-Lieb-like
relations, a fact that led to the definition of the ``Jones polynomial'' of an oriented link. In the notation of 
\cite[p. 104, (I)--(VI)]{J3}, Jones showed that if $f_i=(q+q\inv)e_i$, then the $f_i$ satisfy the 
relations \eqref{eq:reltl}, where Jones' parameter $t$ is replaced by $q^2$.
If $q^2\neq -1$, Jones' form on $\TL_n(q)$ is defined as
 the unique (invariant) trace $\tr_n=\tr$ on $\TL_n(q)$ which satisfies
 \be\label{eq:jf}
 \tr(1)=1\text {   and   } \tr(xf_i)=-(q+q\inv)\inv\tr(x)\text{   for   }x\in\TL_{i-1}\subset\TL_i\subseteq\TL_n
 \ee
 for $1\leq i\leq n-1$.
 
 This trace on $\TL_n(q)$ is non-degenerate if and only if $q^2$ is not a root of unity,
 or, if $|q^2|=\ell$, $n\leq \ell-2$ (\cite[(3.8)]{GLA}. Thus the discrete set of values of $q^2$ for which
 the Jones' sequence $(A_n)$ of algebras is infinite coincides precisely with the set of values of $q^2$
 for which the trace form above on $\TL_n(q)$ is degenerate.  
 We are concerned in this work with the structure and representations of the quotient of $\TL_n(q)$ 
 by the radical of the Jones trace form.

 \subsection{Cell modules and forms}\label{ss:cmod} Let us fix $n$ and consider the representation theory of $\TL_n$.
 By \cite{GL} or \cite{GLA}, $\TL_n$ has cell modules $W_t:=W_t(n)$ whose basis is the set of monic
 Temperly-Lieb morphisms from $t$ to $n$, where $t\in\CT(n)$, and 
 $\CT(n)=\{t\in\Z\mid 0\leq t\leq n\text{ and }t+n\in2\Z\}$.
 
 Now $W_t$ has an invariant form $(\;,\;)$ which may be described as follows. For monic diagrams
 $D_1,D_2:t\to n$, we form the diagram $D_1^*D_2:t\to t$. If $D_1^*D_2$ is monic (i.e. a multiple of $\id_t$), then we write
 $D_1^*D_2=(D_1,D_2)\id_t$; otherwise we say $(D_1,D_2)=0$. 
 Here $D^*$ denotes the diagram obtained from $D$ by reflection in a horizontal, extended to $W_t$ by linearity.

 \section{Semisimplicity and non-degeneracy.} 
 Clearly, if the trace \eqref{eq:jf} is non-degenerate, the algebra $\TL_n$ is semisimple. The converse is true
except for one single case  (see \cite[Rem. 3.8, p.204]{GLA}). It follows from \cite[Cor. (3.6)]{GLA}
that if $|q^2|=\ell$, $\TL_n$ is non-semisimple if and only if $n\geq \ell$. 
Moreover we have very precise information concerning the radical of 
 the invariant trace form (which we henceforth call the Jones form).

\subsection{Radical of the trace form} The  radical of the trace form above is given by the following result
(see \cite[\S 3]{GLA}, \cite{J1}).
\begin{proposition}\label{prop:rad}
If $q$ is not a root of unity then $\tr$ is non-degenerate and $\TL_n$ is semisimple for all $n$.

Suppose the order of $q^2$ is $\ell$. Then there is a unique idempotent $E_{\ell-1}\in\TL_{\ell-1}$ (the Jones-Wenzl idempotent)
such that $f_iE_{\ell-1}=E_{\ell-1}f_i=0$ for $1\leq i\leq \ell-2$.
Moreover for $n\geq \ell$ the radical of $\tr_n$ is generated as ideal of $\TL_n$ by $E_{\ell-1}$.
\end{proposition}

\begin{remark}cf. \cite[Remark (3.8)]{GLA}
It follows from Proposition \ref{prop:rad} that the trace $\tr_n$ is non-degenerate if and only if 
$n\leq \ell-2$, where $\ell=|q^2|$. It follows that the case $n=\ell-1$ is uniquely characterised as 
the one where the form $\tr$ is degenerate, but $\TL_{\ell-1}$ is semisimple.
\end{remark}

The following formula for the idempotent $E_{\ell-1}$ was proved in \cite[Cor. 3.7]{GLA}.
To prepare for its statement, recall that if $F$ is a finite forest (i.e. a partially ordered set in which 
$x\leq a, x\leq b\implies a\leq b$ or $b\leq a$), then we define a Laurent polynomial

\be\label{eq:hf}
h_F(x)=\frac{[|F|]_x!}{\prod_{a\in F}[|F_{\leq a}|]_x},
\ee
where, for $m\in\N$, $[m]_x=\frac{x^m-x^{-m}}{x-x\inv}$ and $[m]_x!=[m]_x[m-1]_x\dots[2]_x[1]_x$.

\begin{theorem}\label{thm:jwi} For any Temperley-Lieb diagram $a:0\to 2n$ we have an associated forest $F_a$, which is simply the 
poset of arcs, ordered by their nesting. For any Temperley-Lieb diagram $D:t\to n$, one obtains a unique diagram $\ol D:0\to t+n$ by rotating 
the bottom line clockwise by $\pi$. With this notation, if $|q^2|=\ell$, we have
\be
E_{\ell-1}=\sum_D h_{F_{\ol D}}(q) D,
\ee
where the sum is over the diagrams from $\ell-1$ to $\ell-1$, i.e. over the diagram basis of $\TL_{\ell-1}$.
\end{theorem}

\begin{example}
If $\ell=4$, we may take $q=-\exp{\frac{\pi i}{4}}$, so that $q^2=i$ and the element $E_3\in\TL_3$ is easily shown to be equal to 
$$
E_3=1+f_1f_2+f_2f_1-\sqrt 2(f_1+f_2).
$$
Note that our defining parameter for $\TL_n$ in this case is $-(q+q\inv)=\sqrt 2$, and the above
element is the familiar one which occurs in the study of the two-dimensional Ising lattice model.
\end{example}

\subsection{The Jones quotient} We now wish to consider the quotient of $\TL_n$ by the ideal generated 
 by $E_{\ell-1}$.
\begin{definition}\label{def:q}
Assume that $|q^2|=\ell$ for a fixed integer $\ell\geq 3$. Let $J_\ell=\langle E_{\ell-1}\rangle$ be the ideal of 
$\TL_n(q)$ generated by the idempotent $E_{\ell-1}\in\TL_{\ell-1}(q)$, where $\TL_{\ell-1}(q)$ is thought of as a
subalgebra of $\TL_n(q)$ for $n\geq\ell-1$ in the obvious way.

The algebra $Q_n=Q_n(\ell)$ ($n=\ell-1,\ell,\ell+1,\dots$) is
defined by
\[
Q_n(\ell)=\frac{\TL_n}{J_\ell}.
\]
This algebra will sometimes be referred to as the ``Jones projection algebra''.
\end{definition}

Since we are taking the quotient by the radical of the trace form $\tr_n$, it follows that $Q_n$ has a non-degenerate invariant 
 trace, and hence that
 
\be\label{eq:qss}   Q_n \text{ is semisimple}.
\ee

\begin{remark}\label{rem:l3} We remark that $E_2=1-f_1$, from which it follows that $Q_n(3)\cong \C$ for all $n$. We therefore henceforth
assume that $\ell=4,5,6,\dots$.
\end{remark}

\section{Representation theory of $Q_n$.}
We wish to understand the simple $Q_n$-modules, and ultimately,  a ``fusion rule'', which describes their tensor products.
This will be approached through the cellular theory, which applies to the algebra $\TL_n$.

\subsection{Review of the representation theory of $\TL_n$ at a root of unity} Let $|q^2|=\ell$. As noted in
Remark \ref{rem:l3} it suffices to consider $\ell\geq 4$; further, since $Q_n(\ell)=\TL_n(q)$ for $n<\ell-1$,
we generally assume that $n\geq \ell-1$. 

Since $\TL_n(q)$ is cellular, it has cell modules $W_t(n)$, 
$t\in\CT(n)$ (cf. \S\ref{ss:cmod}).
These have basis the set of monic diagrams $:t\to n$, are generically simple, but occasionally have two composition factors.
As mentioned above, $W_t(n)$ has an invariant form $(\;,\;)_t$; let $L_t:=W_t/\Rad_t$, where $\Rad_t$ is the radical of the form
$(\;,\;)_t$. Then the general theory asserts that the $L_t$ are simple, and represent all the distinct isomorphism classes of
simple $\TL_n(q)$-modules.

The following description of the composition factors of $W_t=W_t(n)$ was given in \cite[Thm. 5.3]{GLA}, and in the formulation here in 
\cite[Thm. 6.9]{ALZ}. 

\begin{theorem}\label{thm:tlcomp}
Let $|q^2|=\ell$, fix $n\geq\ell$ and let $\CT(n)$ be as above. Let $\N'=\{i\in\N\mid i\not\equiv -1(\text{mod } \ell)\}$.
Define $g:\N'\to\N'$ as follows: for $t=a\ell+b\in\N'$, $0\leq b\leq \ell-2$, define $g(t)=(a+1)\ell+\ell-2-b$. Notice that 
$g(t)-t=2(\ell-b-1)$, so that $g(t)\geq t+2$ and $g(t)\equiv t(\text{mod } 2)$.
\begin{enumerate}
\item For $t\in\CT(n)\cap\N'$ such that $g(t)\in\CT(n)$, there is a non-zero homomorphism $\theta_t:W_{g(t)}(n)\to W_t(n)$.
These are explicitly described in \cite[Thm 5.3]{GLA}, and are the only non-trivial homomorphisms between the cell modules
of $\TL_n$.
\item If $t\in\CT(n)$ is such that $t\in\N'$ and $g(t)\in\CT(n)$, then $W_t(n)$ has composition factors $L_t$ and $L_{g(t)}$,
each with multiplicity one. All other cell modules for $\TL_n(q)$ are simple.
\item If $\ell\geq 3$, all the modules $L_t$, $t\in\CT(n)$, are non-zero, and form a complete set of simple $\TL_n(q)$-modules. 
\end{enumerate}
\end{theorem} 

\subsection{Representation theory of $Q_n$} We have seen \eqref{eq:qss} that the algebras $Q_n$ are semisimple.
We therefore focus on the description of their simple modules.

\begin{proposition}\label{prop:sim} 
Let $n\geq\ell-1$. The simple $Q_n$-modules are precisely those simple $\TL_n$-modules $L_t$, $t\in\CT(n)$, such that 
$J_\ell L_t=0$, where $J_\ell$ is the ideal of $\TL_n$ generated by $E_{\ell-1}$.
\end{proposition}
\begin{proof}
Recall that $Q_n=\TL_n/J_{\ell}$.
 If $M$ is any $Q_n$-module, it may be lifted via the 
canonical surjection $\TL_n\overset{\eta_n}{\lr} Q_n$ to a $\TL_n$-module which we denote by $\wt M$, on which
$J_{\ell}$ acts trivially. Conversely, if $\wt M$ is any $\TL_n$-module on which $J_\ell$ acts trivially, the action factors
through $\TL_n/J_{\ell}=Q_n$, so that $\wt M$ may be thought of as a $Q_n$-module $M$.

Moreover, it is clear that $M$ is simple as $Q_n$-module if and only if $\wt M$ is simple as $\TL_n$-module.

Suppose now that $M$ is a simple $Q_n$-module. Then $\wt M$ is a simple $\TL_n$-module, and hence by
 Theorem \ref{thm:tlcomp} (3), is isomorphic to $L_t$ for some $t\in\CT(n)$ and by the above remarks, $J_\ell$
 acts trivially on $L_t$. Conversely, if $L_t$ satisfies $J_\ell L_t=0$, then $L_t$ is a simple $Q_n$-module. 
 \end{proof}
 
 \begin{remark}\label{rem:zero}
 If $N$ is a $\TL_n$ module, then since $J_\ell=\TL_n E_{\ell-1}\TL_n$,
 it follows that $J_\ell N=0$ if and only if $E_{\ell-1}N=0$. Thus the condition in the Proposition
 is relatively straightforward to check.
 \end{remark}
 
  \begin{remark}[Remark concerning notation]\label{rem:zero} Although {\it a priori} $E_{\ell-1}\in\TL_{\ell-1}$,
  we have regarded it as an element of $\TL_n$ for any $n\geq\ell-1$. The strictly correct notation for $E_{\ell-1}\in\TL_n$,
where $n\geq\ell$, is $E_{\ell-1}\ot I^{\ot (n-\ell+1)}$, where the tensor product is in the Temperley-Lieb category $\bT$,
as described in \cite{GLA} or \cite{LZ5}; that is, it is described diagrammatically as juxtaposition of diagrams, and $I$ is the
identity diagram from $1$ to $1$. We shall use this notation freely below.
  \end{remark}
  
  \begin{theorem}\label{thm:nz} With notation as in Theorem \ref{thm:tlcomp}, let $t\in\CT(n)$ satisfy $t\geq \ell-1$.
  Then the idempotent $E_{\ell-1}\ot I^{\ot (n-\ell+1)}$ acts non-trivially on $L_t$. Thus $Q_n$ has at most $[\frac{\ell}{2}]$
  isomorphism classes of simple modules.
  \end{theorem}
  \begin{proof}
  We begin by showing that
  if $t\in\CT(n)$ and $t\geq \ell-1$ then $E_{\ell-1}\ot I^{\ot (n-\ell+1)}W_t$ contains  all diagrams of the
 form $I^{\ot t}\ot D'$, where $D'$ is any monic diagram from $0$ to $n-t$.

  To see this, note that $W_t$ is spanned by monic diagrams from $t$ to $n$ in $\bT$. 
Take $D=I^{\ot t}\ot D'\in W_t$, where $D'$ is any (monic) diagram
from $0$ to $n-t$.   By the formula in Theorem \ref{thm:jwi}, the coefficient of $I^{\ot(\ell-1)}$ in $E_{\ell-1}$
is $1$. Since all the other summands act trivially on $D$ (because they reduce the number of `through strings'), it follows that 
$E_{\ell-1}\ot I^{\ot (n-\ell+1)}D=D$ in $W_t$, and hence that $D\in E_{\ell-1}\ot I^{\ot (n-\ell+1)}W_t$.

Now if $D=I^{\ot t}\ot D'$ as above and $\phi_t$ is the canonical bilinear form on $W_t$ (see \cite[\S 2]{GL}), then 
$\phi_t(D,D)$ is a power of $-(q+q\inv)$, and hence is non-zero. It follows from \cite[Cor. (2.5)(ii)]{GL} that $W_t=\TL_n D$,
whence $W_t=J_\ell W_t$. Moreover it also follows that $D\not\in\Rad_t$, and hence that modulo $\Rad_t$, $D$ generates
$L_t$, whence $J_\ell L_t=L_t$.
  \end{proof}
  
It follows from the above result that the only possible simple $Q_n$-modules are the $L_t$ with $t< \ell-1$. 
\begin{theorem}\label{thm:zero}
The simple $Q_n$ modules are the $L_t$ with $t\leq \ell-2$. 
\end{theorem}
\begin{proof}
In view of Proposition \ref{prop:sim} and Theorem \ref{thm:nz}, it suffices to show that 
$(E_{\ell-1}\ot I^{\ot (n-\ell+1)})L_{\ell-2}= 0$ and $(E_{\ell-1}\ot I^{\ot (n-\ell+1)})W_t=0$
for $t\leq \ell-3$.

Consider first the case $t\leq \ell-3$. Then any monic diagram $D:t\to n$ contains an upper horizontal arc 
whose right vertex $\leq \ell-1$. Since $E_{\ell-1}$ is harmonic in $\TL_{\ell-1}$, 
it follows that $(E_{\ell-1}\ot I^{\ot (n-\ell+1)})D=0$, and hence that 
$(E_{\ell-1}\ot I^{\ot (n-\ell+1)})W_t=0$.

Now consider the case $t=\ell-2$. Then $g(t)=\ell$ (see the statement of Theorem \ref{thm:tlcomp}) and it follows
from Theorem \ref{thm:tlcomp}(2) and the fact that $n\geq \ell$, that $W_{\ell-2}(n)$ has composition factors
$L_{\ell-2}$ and $L_\ell$. If $D:t\to n$ is a monic diagram, then by the harmonic nature of $E_{\ell-1}$,
$(E_{\ell-1}\ot I^{\ot (n-\ell+1)})D=0$ unless $D=I^{\ot (\ell-2)}\ot D'$, where $D'$ is a diagram from $0$ 
to $n-\ell+2$. But in this case an inspection of the diagrams shows that if $(E_{\ell-1}\ot I^{\ot (n-\ell+1)})D=x\in W_{\ell-2}$,
then $\phi_{\ell-2}(x,x)$ is a multiple of $\tr_{\ell-1}(E_{\ell-1})=0$. 

More generally, if $D_1,D_2$ are diagrams
in $W_{\ell-2}$ and $(E_{\ell-1}\ot I^{\ot (n-\ell+1)})D_i=x_i$ ($i=1,2$), then the same argument shows that 
$\phi_{\ell-2}(x_1,x_2)=0$. 

It follows that for any diagrams $D_1,D_2\in W_{\ell-2}$, $\phi_{\ell-2}((E_{\ell-1}\ot I^{\ot (n-\ell+1)})D_1,D_2)=0$,
since $E_{\ell-1}\ot I^{\ot (n-\ell+1)}$ is idempotent and self dual, so that 
$$
\begin{aligned}
\phi_{\ell-2}((E_{\ell-1}\ot I^{\ot (n-\ell+1)})D_1,D_2)=&\phi_{\ell-2}((E_{\ell-1}\ot I^{\ot (n-\ell+1)})^2D_1,D_2)\\
=&\phi_{\ell-2}((E_{\ell-1}\ot I^{\ot (n-\ell+1)})D_1,(E_{\ell-1}\ot I^{\ot (n-\ell+1)})D_2)\\
=&0.\\
\end{aligned}
$$

Hence $J_\ell W_{\ell-2}\subseteq\Rad_{\ell-2}$, and it follows that $J_\ell L_{\ell-2}=0$.
\end{proof}

\section{Dimensions.}

In this section we shall discuss the dimensions of the simple $Q_n(\ell)$-modules. To begin with, we give explicit formulae for the
dimensions of the cell modules of $\TL_n$.

\subsection{Cell modules for $\TL_n$} Recall that the cell module $W_t(n)$ has a basis consisting of the monic $\TL$-diagrams $D:t\to n$.
Since such diagrams exist only when $t\equiv n(\text{mod }2)$, we may write $n=t+2k$, $k\geq 0$. 
\begin{definition}
For $t,k\geq 0$, we write $w(t,k):=\dim W_t(t+2k)$. By convention, $W_0(0)=0$, so that $w(0,0)=0$.
\end{definition}

\begin{proposition}\label{prop:recw}
We have the following recursion for $w(t,k)$. For integers $t,k\geq 0$:
\be\label{eq:recw}
w(t,k+1)=w(t-1,k+1)+w(t+1,k).
\ee
\end{proposition}
\begin{proof}
The proof is based on the interpretation of $w(t,k)$ as the number of monic $\TL$-diagrams from $t$ to $t+2k$.

Consider first the case $t=0$. The assertion is then that $w(0,k+1)=w(1,k)$. But all $\bT$-diagrams $D:1\to 1+2k$ are monic,
as are all diagrams $0\to 2\ell$ (any $\ell$). It follows that $w(1,k)=\dim(\Hom_\bT(1,1+2k))=\dim(\Hom_\bT(0,2+2k))=w(0,k+1)$.
Thus the assertion is true for $t=0$ and all $k\geq 0$. Similarly, if $k=0$, the assertion amounts to $w(t,1)=w(t-1,1)+w(t+1,0)$.
If $t>0$, the left side is easily seen to be equal to $t+1$, while $w(t-1,1)=t$ and $w(t+1,0)=1$. If $t=0$, the left side is 
equal to $\dim(\Hom_\bT(0,2k+2))=\dim(\Hom_\bT(1,2k+1))=w(1,k)$. So the recurrence is valid for $k=0$ and all $t$.

Now consider the general case. Our argument will use the fact that $w(t,k+1)$ may be thought of as the number of 
$\TL$-diagrams $0\to 2t+2(k+1)$ of the form depicted in Fig. 1.

\begin{figure}
\begin{tikzpicture}
\foreach \x in {1,2,4,5,6,7, 9,10}
\filldraw(\x,0) circle (0.05cm);
\draw [dashed] (4.5,-3)--(4.5,1);
\node at (1,.5) {$a_1$}; \node at (2,.5) {$a_2$};\node at (4,.5) {$a_t$};
\node at (5,.5) {$b_1$};\node at (6,.5) {$b_2$};\node at (10,.5) {$b_{t+2(k+1)}$};\node at (7,.5) {$b_3$};
\node at (3,0) {$...$};\node at (8,0) {$...$};
\draw (1,0).. controls (3,-3) and (8,-3).. (10,0); 
\draw (2,0).. controls (4,-2) and (7,-2).. (9,0);
\draw (4,0).. controls (5,-1.3) and (6,-1.3).. (7,0);
\draw (5,0).. controls (5.5,-.75) and (5.5,-.75).. (6,0);
%\draw (0,).. controls (0,1) and (3,1).. (3,3);
%\draw (1,3).. controls (1,1) and (5,1).. (5,3);
%\draw (1,3).. controls (2,0) and (5,1).. (5,3);
%\draw (2,3).. controls (2,0) and (6,0).. (6,3);
%\draw (4,3).. controls (4,0) and (8,0).. (8,3);
%\draw (7,3).. controls (7,0) and (12,0).. (12,3);
%\draw (7,3).. controls (7,0) and (10,0).. (10,3);
%\node at (9,2.5) {$\cdots$};
\end{tikzpicture}
\caption{Monic diagram $t\to t+2(k+1)$ as a diagram $0\to 2t+2(k+1)$.}
\end{figure}
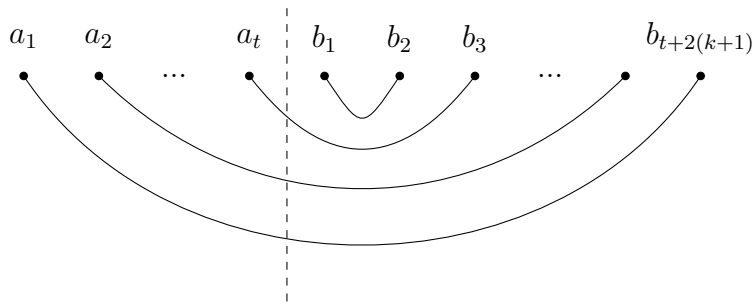

The condition that the diagram be monic is simply that each $a_i$ is joined to some $b_j$, i.e. that each arc crosses the 
dotted line; of course distinct arcs are non-intersecting.

Now evidently such diagrams fall into two types: those in which $[a_t,b_1]$ is an arc, and the others. Now the number of diagrams 
in which $[a_t,b_1]$ is an arc is clearly equal to $w(t-1, k+1)$, while those in which $[a_t,b_1]$ is not an arc are in bijection with
the monic diagrams from $t+1$ to $t+1+2k$, as is seen by shifting the dotted line one unit to the right. Hence the number of the latter 
is $w(t+1,k)$, and the recurrence \eqref{eq:recw} is proved.
\end{proof}

\subsection{A binomial expression for $w(t,k)$} 
\begin{definition}\label{def:bin}
For integers $t,k\geq 0$, define 
\be\label{eq:bin}
F(t,k)=\binom{t+2k}{k}-\binom{t+2k}{k-1}.
\ee
This definition is extended to the domain $\Z\times\Z$ by stipulating that $F(t,k)=0$ if $t<0$ or $k<0$.
\end{definition}

It is easily seen that 
\be\label{eq:bin1}
F(t,k)=\frac{(t+1)(t+2k)(t+2k-1)\dots(t+k+2)}{k!}=\frac{t+1}{t+k+1}\binom{t+2k}{k}.
\ee

\begin{lemma}\label{lem:recf}
We have the following recursion for $F(t,k)$. For $t,k\geq 0$:
\be\label{eq:recf}
F(t,k+1)=F(t-1,k+1)+F(t+1,k).
\ee
\end{lemma}
\begin{proof}
For any integers $r,n\geq 0$ we have the binomial identity $\binom{n}{r+1}=\binom{n-1}{r+1}+\binom{n-1}{r}$.
Using this in the form $\binom{n}{r+1}-\binom{n-1}{r+1}=\binom{n-1}{r}$, we see from \eqref{eq:bin} that
$$
F(t, k+1)-F(t-1, k+1)=\binom{t+2k+1}{k}-\binom{t+2k+1}{k-1}=F(t+1,k).
$$
\end{proof}

\subsection{Catalan calculus--generating functions} For $n\geq 0$, write $c(n):=w(0,2n)$, $c(0)=1$.
It is easily seen by inspecting diagrams that for $n>0$, 
\be\label{eq:recc}
c(n)=\sum_{k=1}^nc(k-1)c(n-k).
\ee
Writing $c(x):=\sum_{n=0}^\infty c(n)x^n$, the recursion \eqref{eq:recc} translates into
\be\label{eq:reccx}
xc(x)^2-c(x)+1=0,
\ee
from which it is immediate that
\be\label{eq:cx}
c(x)=\frac{1-(1-4x)^{\frac{1}{2}}}{2x},
\ee
and applying the binomial expansion, that
\be\label{eq:cn}
c(n)=\frac{1}{n+1}\binom{2n}{n}.
\ee

Now define $W_t(x)=\sum_{k=0}^\infty w(t,k)x^k$. Inspection of diagrams shows that the $w(t,k)$ satisfy the following
recursion.
\be\label{eq:recwc}
w(t,k)=\sum_{\ell=0}^kw(t-1,\ell)c(k-\ell),
\ee
which translates into the recursion $W_t(x)=W_{t-1}(x)c(x)$ for the generating function $W_t(x)$. Using the fact that
evidently $W_0(x)=c(x)$, we have proved the following statement.
\begin{proposition}\label{prop:wx}
For $t=0,1,2,\dots$, we have $W_t(x)=\sum_{k=0}^\infty w(t,k)x^k=c(x)^{t+1}$.
\end{proposition}

\begin{corollary}\label{cor:wxy}
We have the following equation in $\Z[[x,y]]$.
\be\label{eq:wxy}
W(x,y):=\sum_{t,k=0}^\infty w(t,k)y^tx^k=\frac{c(x)}{1-yc(x)}.
\ee
\end{corollary}

\subsection{A closed expression for $w(t,k)$} We shall prove the following theorem.

\begin{theorem}\label{thm:dimw}
For integers $t,k\geq 0$, we have 
\be\label{eq:wf}
w(t,k)=F(t,k).
\ee
 That is,
\be\label{eq:dimw}
\dim(W_t(t+2k))=\frac{t+1}{t+k+1}\binom{t+2k}{k}.
\ee
\end{theorem}
\begin{proof}
Consider first the case $k=0$. Then for any $t\geq 0$, $w(t,0)=1$, while $F(t,0)=\frac{1}{1}\binom{t}{0}=1.$
Thus the result is true for $k=0$ and all $t$. Next consider the case $t=0$. Then for all $k\geq 0$, 
$w(0,k)=c(k)=\frac{1}{k+1}\binom{2k}{k}$, while $w(0,k)$ is evidently also equal to $c(k)$ by  
\eqref{eq:bin1}. Thus the assertion is true for $t=0$ and all $k$.

Now suppose that the following assertion is true.

\noindent {$\mathbf{P(t_0)}$:} {\it The result is true for pairs $(t,k)$ such that $t\leq t_0$ (and any $k\geq 0$).}

We now use the recursions \eqref{eq:recw} and \eqref{eq:recf}, in the following form.
\be\label{eq:comprec}
\begin{aligned}
F(t_0+1,k)=&F(t_0,k+1)-F(t_0-1,k+1)\text{ and }\\
w(t_0+1,k)=&w(t_0,k+1)-w(t_0-1,k+1).\\
\end{aligned}
\ee

By assumption, the right sides of \eqref{eq:comprec} are equal, whence $F(t_0+1,k)=w(t_0+1,k)$, i.e. $\mathbf P(t_0+1)$ 
holds. Since we have seen that $\mathbf P(0)$ is true, the result follows.
\end{proof}

\subsection{Application to the simple modules for $Q_n(\ell)$} Fix $\ell\geq 3$ and assume that $q^2$ has order $\ell$.
We wish to study the dimensions of the simple $Q_n$-modules, which we have already seen (Theorem \ref{thm:zero})
are the simple $\TL_n$-modules $L_t(n)$ for $0\leq t\leq \ell-2$, and $t\equiv n(\text{mod }2)$. 

\begin{definition}\label{def:dims}
It will be convenient to denote the dimensions of the cell modules $W_t(n)$ and simple modules $L_t(n)$ for
the algebra $\TL_n$ by $w_t(n)$ and $\ell_t(n)$ respectively. We regard $w_t$ and $\ell_t$ as functions of $n\in\Z_{\geq 0}$.
These functions satisfy the following conditions: $w_t(n)=0$ if $t>n$ or if $t<0$. More generally, $w_t(n)=\ell_t(n)=0$
for any $t\in\Z$ such that $t\not\in \CT(n)$ (see Theorem \ref{thm:tlcomp}).
\end{definition}

 We have seen by Theorem \ref{thm:tlcomp} that in general, the composition factors of the cell module $W_t(n)$ 
(for $n\in\N'$) are $L_t(n)$ and $L_{g(t)}(n)$ where $g:\N'\to\N'$ is the function defined in 
Theorem \ref{thm:tlcomp}. The following result is an easy consequence of this.

 \begin{proposition}\label{prop:diml} Let $|q^2|=\ell$. 
We have, for $t\in\N'=\{s\in\Z_{\geq 0}\mid s\not\equiv -1(\text{mod }\ell)\}$:
\be\label{eq:diml}
\ell_t(n)=\sum_{i=0}^\infty (-1)^iw_{g^i(t)}(n).
\ee
\end{proposition}
\begin{proof}
Note that since $g$ is a strictly increasing function on $\N'$, for any particular $n$, the sum on the right side
of \eqref{eq:diml} is finite. 

It is evident from Theorem \ref{thm:tlcomp} (2), that for any $t\in\N'$, 
\be\label{eq:lt1}
\ell_t(n)=w_t(n)-\ell_{g(t)}(n).
\ee

Applying \eqref{eq:lt1} with $t$ replaced by $g(t)$ gives $\ell_t(n)=w_t(n)-w_{g(t)}(n)+\ell_{g^2(t)}(n)$.
Applying this repeatedly, and noting that there is an integer $t_0\in\N'$ such that $t_0\leq n$ and $g(t_0)>n$, 
we obtain the relation \eqref{eq:diml}.
\end{proof}

\begin{remark}\label{rem:fns}
\begin{enumerate}
\item In view of the interpretation of $w_t$ and $\ell_t$ as functions (see Definition \ref{def:dims}), the relation 
\eqref{eq:diml} will be written as the following equality of functions on $\N'$.
\be\label{eq:fdim}
\ell_t=\sum_{i=0}^\infty (-1)^iw_{g^i(t)}.
\ee

\item Note that by Theorem \ref{thm:zero}, the simple $Q_n$-modules are those $L_t$ with $0\leq t\leq\ell-2$. Thus 
for these modules, $t\in\N'$, and the analysis above applies to all the simple modules.
\end{enumerate}
\end{remark}

\section{The Ising algebra $Q_n(4)$}\label{sss:4} 

We shall now give a detailed analysis of the particular case $\ell=4$. We refer to $Q_n$ in this case as
the Ising algebra. In this section, $\ell=4$, so that $q^2=\sqrt {-1}$ and we take $q=-\exp{\frac{\pi i}{4}}$. 
The first line of the relation \eqref{eq:reltl} then reads:
\[
f_i^2=\sqrt {2}f_i.
\]

\subsection{The function $g$}
 The function $g(t)$ ($t\in\N$, $t\not\equiv -1\text{ (mod }4)$) defined in Theorem \ref{thm:tlcomp} is easily described in this case. We have

\be\label{eq:g4}
g(t)=
\begin{cases}
t+6  \text{ if } t\equiv 0\text{(mod }4)\\
t+2  \text{ if } t\equiv 2\text{(mod }4)\\
t+4  \text{ if } t\equiv 1\text{(mod }4).\\
\end{cases}
\ee

\subsection{Structure and dimension of $Q_n$}
The main theorem of this section is

\begin{theorem}\label{thm:ising}
Let $\ell=|q^2|=4$ and for $n=3,4,5,6,\dots$, let $Q_n$ be the Ising algebra defined in Definition \ref{def:q}.
The dimensions of the simple $Q_n$-modules are as follows. If $n$ is even, then $Q_n$ has just two simple modules
$L_0$ and $L_2$, whose dimensions $\ell_0(n)$ and $\ell_2(n)$ are given by
\be\label{eq:dimlev}
\ell_0(n)=\ell_2(n)=2^{\frac{n}{2}-1}.
\ee
If $n$ is odd, then $Q_n$ has a unique simple module $L_1(n)$, whose dimension $\ell_1(n)$ is given by
\be\label{eq:dimlod}
\ell_1(n)=2^{\frac{n-1}{2}}.
\ee
\end{theorem}
\begin{proof}
Combining the general formula \eqref{eq:diml} with the particular values \eqref{eq:g4} of the function $g$ in this case,
we obtain the following expressions for the relevant dimensions.
\be\label{eq:dimli4}
\begin{aligned}
\ell_0&=\sum_{t\in\N,\;t\equiv 0\text{(mod }8)} w_t-\sum_{t\in\N,\;t\equiv 6\text{(mod }8)} w_t,\\
\ell_2&=\sum_{t\in\N,\;t\equiv 2\text{(mod }8)} w_t-\sum_{t\in\N,\;t\equiv 4\text{(mod }8)} w_t\text{ and}\\
\ell_1&=\sum_{t\in\N,\;t\equiv 1\text{(mod }8)} w_t-\sum_{t\in\N,\;t\equiv 5\text{(mod }8)} w_t.\\
\end{aligned}
\ee

We shall prove the relations \eqref{eq:dimlev} and \eqref{eq:dimlod} by induction on $n$. Observe that if $n=4$,
$\ell_0(4)=w_0(4)=2$, and $\ell_2(4)=w_2(4)-w_4(4)=3-1=2$, so that \eqref{eq:dimlev} is true for $n=4$. 
Similarly $\ell_1(3)=w_1(3)=2$, so that \eqref{eq:dimlod} holds for $n=3$; moreover $\ell_1(5)=w_1(5)-w_5(5)
=5-1-4$, so that \eqref{eq:dimlod} also holds for $n=5$.

Next, notice that the recursion \eqref{eq:recw} for $w_t(n)$ may be written
\be\label{eq:recw2}
w_t(n+1)=w_{t-1}(n)+w_{t+1}(n) \text{ for }t\geq 0\text{ and } n\in\N'.
\ee
Applying the same recursion to both terms on the right side
 of \eqref{eq:recw2} and replacing $n$ by $n+1$, we obtain
\be\label{eq:recw3}
w_t(n+2)=
\begin{cases}
w_{t-2}(n)+2w_t(n)+w_{t+2}(n)\text{ if }t\geq 1\\
w_{t}(n)+w_{t+2}(n)\text{ if }t=0.\\
\end{cases}
\ee
Note that the relation \eqref{eq:recw3} holds for all $n\in\Z$ and $t$ in the stated range, subject to the interpretation of $w_t$ 
given in Definition \ref{def:dims}.

Let us now compute $\ell_1(n)$. Since the evaluation of the infinite sums in \eqref{eq:dimli4} 
at any $n\in\N'$ involves only a finite sum, we may rearrange the terms in any way we choose.
The third line of \eqref{eq:dimli4} may therefore be written as follows. 
\be\label{eq:l1-4}
\ell_1=\sum_{t\in \N,\; t\equiv 1\text{(mod }8)}(w_t-w_{t+4}).
\ee

But it follows from two applications of the first line of \eqref{eq:recw3} that for any 
odd integer $n\geq 3$ and $t\geq 1$, we have 
\be\label{eq:red1-4}
(w_t-w_{t+4})(n+2)=(w_{t-2}-w_{t+6})(n)+2(w_t-w_{t+4})(n).
\ee

Summing \eqref{eq:red1-4} over $t\geq 0$ such that $t\equiv 1\text{(mod }8)$, we obtain, using \eqref{eq:l1-4},
\[
\ell_1(n+2)=\sum_{t\geq 1,\;t\equiv 1\text{(mod }8)}\left(w_{t-2}(n)-w_{t+6}(n)\right)+2\ell_1(n).
\]
Moreover $\sum_{t\geq 1,\;t\equiv 1\text{(mod }8)}\left(w_{t-2}(n)-w_{t+6}(n)\right)=w_{-1}(n)+w_7(n)-w_7(n)
+w_{15}(n)-w_{15}(n)+\dots=0$. Hence $\ell_1(n+2)=2\ell_1(n)$ for $n\geq 3$, and since $\ell_0(3)=2$,
it follows by induction that $\ell_1(2r+1)=2^r$ ($r\geq 1$), as asserted in \eqref{eq:dimlod}.

Next consider $\ell_2(n)$ for $n\geq 4$ even. The second line of  \eqref{eq:dimli4} may be written
\be\label{eq:l2-4}
\ell_2=\sum_{t\in \N,\; t\equiv 2\text{(mod }8)}(w_t-w_{t+2}).
\ee

A similar argument to that above shows that for $t\geq 2$, 
\be\label{eq:red2-4}
(w_t-w_{t+2})(n+2)=(w_{t-2}-w_{t+4})(n)+(w_t-w_{t+2})(n),
\ee
and summing over $t$ such that $t\equiv 2\text{(mod }8)$, we see that 
\be\label{eq:recl2}
\ell_2(n+2)=\ell_0(n)+\ell_2(n).
\ee

Now consider $\ell_0$. In this case, we have 
\be\label{eq:l0-4}
\ell_0=\sum_{t\in \N,\; t\equiv 0\text{(mod }8)}(w_t-w_{t+6}).
\ee

In view of \eqref{eq:recw3} we need to treat the case $t=0$ separately. 
Arguing as above, we obtain
\be\label{eq:red0-4}
(w_t-w_{t+6})(n+2)=
\begin{cases}
(w_{t-2}-w_{t+6})(n)+(w_t-w_{t+8})(n)+(w_t-w_{t+6})(n)\\
\hfill+(w_{t+2}-w_{t+4})(n)\text{ if }t\geq 2,\\
(w_0+w_2-w_4-2w_6-w_8)(n)\text { if }t=0.\\
\end{cases}
\ee

One now performs the straightforward but detailed task of computing
the sum $\sum_{t\in \N,\; t\equiv 0\text{(mod }8)}(w_t-w_{t+6})(n+2)$ using \eqref{eq:red0-4}.
This shows that
\be\label{eq:recl0}
\ell_0(n+2)=\ell_0(n)+\ell_2(n).
\ee

Since $\ell_0(4)=\ell_2(4)=2$, the relations \eqref{eq:recl2} and \eqref{eq:recl0}
complete the proof of the Theorem by induction.
\end{proof}

\begin{corollary}\label{cor:dimq-4}
Let $\ell=|q^2|=4$. The dimension of the Ising algebra $Q_n$ defined in Definition \ref{def:q}
is $2^{n-1}$.
\begin{proof}
We have seen that $Q_n$ is a semisimple finite dimensional algebra. Its dimension is therefore the sum of the squares
of the dimensions of its distinct simple modules. The statement now evidently follows from Theorem \ref{thm:ising}.
\end{proof}
\end{corollary}

\subsection{The Clifford algebra and $Q_n$} Let $U$ be a complex vector space of finite dimension $n$, with a non-degenerate 
symmetric bilinear form $\langle -,-\rangle$. Then $U$ has an orthonormal basis $u_1,\dots,u_n$, which enjoys 
the property that $\la u_i,u_j\ra=\delta_{ij}$. If  $\gamma_i=\frac{1}{\sqrt {2}}u_i$ for $i=1,\dots,n$, then for any $i,j$, 
\be\label{eq:or}
\la \gamma_i,\gamma_j\ra=\frac{1}{2}\delta_{i,j}.
\ee

The Clifford algebra $\cC_n=\cC(U,\la-,-\ra)$ (for generalities about Clifford algebras we refer the reader to \cite{DG}) is defined as
\be\label{eq:cl}
\cC_n=\frac{T(U)}{I},
\ee
where $T(U)=\oplus_{i=0}^\infty U^{\ot i}$ is the free associative $\C$-algebra (or tensor algebra) on $U$, and $I$ is the ideal
of $T(U)$ generated by all elements of the form $u\ot u-\la u,u\ra 1$ ($u\in U$). This last relation may equivalently be written 
(omitting the $\ot$ in the multiplication)
\be\label{eq:clrel}
uv+vu=2\la u,v\ra 1.
\ee

The algebra $\cC_n$ is evidently generated by any basis of $U$, and hence by \eqref{eq:or} and \eqref{eq:clrel} has the presentation
\be\label{eq:prescl}
\cC_n=\la \gamma_1,\dots,\gamma_n\mid \gamma_i\gamma_j +  \gamma_j\gamma_i=\delta_{ij}\text{ for }1\leq i,j\leq n \ra.
\ee

For any subset $J=\{j_1<j_2<\dots<j_p\}\subseteq \{1,\dots,n\}$, write $\gamma_J=\gamma_{j_1}\gamma_{j_2}\dots\gamma_{j_p}$.
It is evident that $\{\gamma_J\mid J\subseteq \{1,2,\dots,n\}\}$ is a basis of $\cC_n$ which is therefore $\Z_2$-graded (since the relations
are in the even subalgebra of the tensor algebra), 
the even (resp. odd) subspace being spanned by those $\gamma_J$ with $|J|$ even (resp. odd).

The following statement is now clear.

\begin{proposition}\label{prop:clf}
Th Clifford algebra $\cC(U,\la-,-\ra)$ has dimension $2^n$, where $n=\dim (U)$. Its even subalgebra $\cC_n^0$ has dimension $2^{n-1}$.
\end{proposition}

The next theorem is the main result of this section; it asserts that the Ising algebra is isomorphic to
the even subalgebra of the Clifford algebra.

\begin{theorem}\label{thm:qcl}
We continue to assume $\ell=4$ and that $q=-\exp(\frac{\pi i}{4})$. Other notation is as above.
For $n=3,4,\dots$ there are surjective homomorphisms $\phi_n:\TL_n(q)\to \cC_n^0$ which induce 
isomorphisms $\ol\phi_n:Q_n\overset{\simeq}{\lr}\cC^0_n$. 
\end{theorem}
\begin{proof} Define $\phi_n(f_j)=\frac{1}{\sqrt 2}(1+2i\gamma_j\gamma_{j+1})$. It was remarked by 
Koo and Saleur \cite[\S 3.1 eq. (3.2)]{KS} (see also \cite{CE}) that the  $\phi_n(f_j)$ satisfy the relations
\eqref{eq:reltl} in $\cC_n$, and therefore that $\phi_n$ defines a homomorphism from $\TL_n$ to $\cC_n$,
and further that $E_3\in\ker(\phi_n)$. 

It is evident that the image of $\phi_n$ is $\cC_n^0$, and therefore that $\ol\phi_n:Q_n\to \cC_n^0$ is surjective.
But by Cor. \ref{cor:dimq-4} and Prop. \ref{prop:clf} these two algebras have the same dimension, whence
$\ol\phi_n$ is an isomorphism.
\end{proof}

\subsection{Canonical trace} Let $\TL_n(q)$ be the $n$-string Temperley-Lieb algebra as above, and assume $\delta:=-(q+q\inv)\neq 0$
is invertible. 
The canonical Jones trace $\tr_n$ on $\TL_n(q)$ is characterised by the properties 
\be\label{eq:proptr}
\begin{aligned}
\tr_n(1)=&1,\\
\tr_{n+1}(xf_n)=&\delta\inv\tr_n(x)\text{ for }x\in\TL_n(q).\\
\end{aligned}
\ee

As has been pointed out above \eqref{eq:qss}, this trace descends to a non-degenerate trace on $Q_n$, satisfying similar properties.

\begin{proposition}
There is a canonical trace $\ol\tr_n$ on $\cC_n^0$, given by taking the constant term (coefficient of $1$) of any of its elements.
This trace corresponds to the Jones trace above in the sense that for $x\in Q_n$, $\tr_n(x)=\ol\tr_n(\phi(x))$. It is therefore non-degenerate.
\end{proposition}

The proof is easy, and consists in showing that $\ol\tr_n$ satisfies the analogue of \eqref{eq:proptr} in $\cC_n^0$.

\subsection{The spinor representations of $\fso(n)$} We give yet another interpretation of the algebra in terms 
of the spin representations of $\fso(n)$. Let $\SO(n)$ be the special orthogonal group of the space $(U,\la-,-\ra)$
above. Its Lie algebra has basis the set of matrices (with respect to the orthogonal basis $(\gamma_i)$) 
$J_{ij}:=E_{ij}-E_{ji}$, $1\leq i<j\leq n$, where the $E_{ij}$ are the usual matrix units. This basis of $\fso(n)$
satisfies the commutation relations
\be\label{eq:com-e}
[J_{ij},J_{kl}]=\delta_{jk}J_{il}-\delta_{jl}J_{ik}-\delta_{ik}J_{jl}+\delta_{il}J_{jk}.
\ee
\begin{proposition}\label{prop:son}
For $n\geq 2$, there are surjective homomorphisms $\psi_n:\U(\fso(n))\to \cC_n^0\cong Q_n$, such that 
$\psi_n(J_{ij})=\omega_{ij}:=\frac{1}{2}(\gamma_i\gamma_j-\gamma_j\gamma_i)$. The irreducible spin representations
of $\fso(n)$ are realised on the simple $Q_n$-modules $L_0$ and $L_2$ when $n$ is even and on $L_1$ when $n$ is odd.
\end{proposition}
\begin{proof} As this is well known, we give merely a sketch of the argument.
To show that $\psi_n$ defines a homomorphism, it suffices to observe that the $\omega_{ij}$ satisfy the same commutation
relations \eqref{eq:com-e} as the $J_{ij}$, and this is straightforward. The surjectivity of $\psi_n$ is evident from the observation
that $\omega_{ij}=\gamma_i\gamma_j$, which shows that the image of $\psi_n$ contains the whole of $\cC_n^0\simeq Q_n$.
\end{proof}

\section{The algebra $Q_n(\ell)$ for $\ell>4$.} The Definition \ref{def:q} defines a quotient $Q_n=Q_n(\ell)$ of $\TL_n(q)$ which depends 
on the order $\ell=|q^2|$. As we shall be considering various $\ell$, we henceforth consistently denote this quotient by $Q_n(\ell)$. Thus the
Ising algebra $Q_n$ of the last section is denoted $Q_n(4)$. If we take $q^2$ to be $\exp(\frac{2\pi i}{\ell})$, then $q=\pm\exp(\frac{\pi i}{\ell})$,
and the constant $\delta$ such that $f_j^2=\delta f_j$ may be taken to be
$\delta=\exp(\frac{\pi i}{\ell})+\exp(-\frac{\pi i}{\ell})=2\cos(\frac{\pi}{\ell})$. In this section we study the simple modules $L_0(n),
L_1(n),\dots,L_{\ell-2}(n)$ of $Q_n(\ell)$. Note that $L_i(n)$ makes sense only if $n\equiv i\text{ (mod } 2)$.
To illustrate the method, we begin by discussing the case $\ell=5$.

\subsection{The algebras $Q_n(5)$} An elementary study of the $g$ function on 
$\N'=\{n\in\N\mid n\not\equiv -1\text{ (mod }\ell)\}$ leads, just as in \eqref{eq:dimli4} to the following formulae for 
the dimension functions $\ell_i$.
\be\label{eq:dimli5}
\begin{aligned}
\ell_0=&\sum_{t\equiv 0\text{(mod }10)}w_t-\sum_{t\equiv 8\text{(mod }10)}w_t
=&\sum_{t\equiv 0\text{(mod }10)}(w_t-w_{t+8});\\
\ell_1=&\sum_{t\equiv 1\text{(mod }10)}w_t-\sum_{t\equiv 7\text{(mod }10)}w_t
=&\sum_{t\equiv 1\text{(mod }10)}(w_t-w_{t+6});\\
\ell_2=&\sum_{t\equiv 2\text{(mod }10)}w_t-\sum_{t\equiv 6\text{(mod }10)}w_t
=&\sum_{t\equiv 2\text{(mod }10)}(w_t-w_{t+4});\\
\ell_3=&\sum_{t\equiv 3\text{(mod }10)}w_t-\sum_{t\equiv 5\text{(mod }10)}w_t
=&\sum_{t\equiv 3\text{(mod }10)}(w_t-w_{t+2}).\\
\end{aligned}
\ee

We next make repeated use of the recurrence \eqref{eq:recw2} to evaluate the right sides of the expressions in \eqref{eq:dimli5}.

Note that 
\[
\begin{aligned}
w_t(n+1)-w_{t+2}(n+1)=& w_{t-1}(n)+w_{t+1}(n)-w_{t+1}(n)-w_{t+3}(n)\\
=&w_{t-1}(n)-w_{t+3}(n),\\
\end{aligned}
\]
and summing both sides over $t\equiv 3\text{(mod }10)$, we obtain

\be\label{eq:5-1}
\ell_3(n+1)=\ell_2(n)\text{ for even }n.
\ee

Similarly, 
\[
\begin{aligned}
w_t(n+1)-w_{t+4}(n+1)=& w_{t-1}(n)+w_{t+1}(n)-w_{t+3}(n)-w_{t+5}(n)\\
=&w_{t-1}(n)-w_{t+3}(n),\\
\end{aligned}
\]
and summing over $t\equiv 2\text{(mod }10)$, we obtain

\be\label{eq:5-2}
\ell_2(n+1)=\ell_1(n)+\ell_3(n)\text{ for odd }n.
\ee

Proceeding further along these lines yields the relations
\be\label{eq:5-3}
\ell_1(n+1)=\ell_0(n)+\ell_2(n)\text{ for even }n.
\ee

and finally

\be\label{eq:5-4}
\ell_0(n+1)=\ell_1(n)\text{ for odd }n.
\ee

Now let $M=\begin{pmatrix} 2 & 1\\1 & 1\\ \end{pmatrix}.$ 
Combining the equations \eqref{eq:5-1}, \eqref{eq:5-2},\eqref{eq:5-3} and \eqref{eq:5-4}, we obtain
\be\label{eq:recq5}
\begin{aligned}
\begin{pmatrix}
\ell_2(n+2)\\ \ell_0(n+2)\\
\end{pmatrix}
&=M\begin{pmatrix}
\ell_2(n)\\ \ell_0(n)\\
\end{pmatrix}\text{ for even }n\text{ and }\\
\begin{pmatrix}
\ell_1(n+2)\\ \ell_3(n+2)\\
\end{pmatrix}
&=M\begin{pmatrix}
\ell_1(n)\\ \ell_3(n)\\
\end{pmatrix}\text{ for odd }n.\\
\end{aligned}
\ee

Now for $n=1,2,3,\dots$, define integers $a_n$ and $b_n$ by 
\be\label{eq:mpn}
\begin{pmatrix} 2 & 1\\1 & 1\\ \end{pmatrix}^n=
\begin{pmatrix}
a_n+b_n & a_n\\a_n & b_n\\
\end{pmatrix}.
\ee

\begin{remark}
Note that the sequence $b_1,a_1,b_2, a_2,\dots=1,1,2,3,\dots$ is the Fibonacci sequence. Let us write 
this sequence $F_1,F_2,\dots$. Then for $n\geq 1$, $b_n=F_{2n-1}$ and $a_n=F_{2n}$. Thus we may write
\be\label{eq:fib1}
\begin{pmatrix} 2 & 1\\1 & 1\\ \end{pmatrix}^n=
\begin{pmatrix}
F_{2n+1} & F_{2n}\\F_{2n} & F_{2n-1}\\
\end{pmatrix}.
\ee
\end{remark}

\begin{proposition}\label{prop:q5}
(1) The dimensions of the simple $Q_n(5)$-modules are given by the following formulae.
For $n=1,2,3,\dots$ write $F_n$ for the $n^{\text{th}}$ term in the Fibonacci sequence. Then
\be
\begin{pmatrix}
\ell_1(2n+1) & \ell_2(2n)\\ \ell_3(2n+1) & \ell_0(2n)\\
\end{pmatrix}
= M^n.
\ee

That is, $\ell_0(2n)=b_n=F_{2n-1}$, $\ell_2(2n)=a_n=F_{2n}$, $\ell_1(2n+1)=a_n+b_n=F_{2n+1}$ and $\ell_3(2n+1)=a_n=F_{2n}$.

(2) The dimensions of the semisimple algebras $Q_n(5)$ are given by 
\[
\dim(Q_n(5))=F_{2n-1},
\]
where $F_n$ is the $n^\text{th}$ Fibonacci number.
\end{proposition}
\begin{proof}
The proof of the first statement is a simple induction on $n$ using the recurrences above, the values for small $n$ being easily calculated.

The second statement follows by taking the sum of the squares of the dimensions of the simple $Q_n(5)$-modules. Specifically,
we have $\dim(Q_{2n}(5))=b_n^2+a_n^2$. But from \eqref{eq:mpn}
\be\label{eq:fib2}
\begin{pmatrix} 2 & 1\\1 & 1\\ \end{pmatrix}^{2n}=
\begin{pmatrix}
a_n+b_n & a_n\\a_n & b_n\\
\end{pmatrix}^2=
\begin{pmatrix}
a_{2n}+b_{2n} & a_{2n}\\a_{2n} & b_{2n}\\
\end{pmatrix},
\ee
so that $a_n^2+b_n^2=b_{2n}$. A similar argument for $\dim(Q_{2n+1}(5))$ shows that
\be
\begin{aligned}
\dim(Q_{2n}(5))=& b_{2n}\text{ and }\\
\dim(Q_{2n+1}(5))=& a_{2n}+b_{2n}.\\
\end{aligned}
\ee 
In view of  \eqref{eq:fib1} and \eqref{eq:fib2}, the proof is complete.
\end{proof}

\subsection{The general case} We now consider $Q_n(\ell)$, where $\ell=|q^2|\geq 4$. By Theorem \ref{thm:zero}, the 
simple $Q_n(\ell)$-modules are those $L_t=L_t(n)$ where $0\leq t\leq \ell-2$ and $i\equiv n\text {(mod }2)$. As usual, we write
$\ell_i(n)=\dim(L_i(n))$ and regard $\ell_i$ as a function on the positive integers. We shall determine the $\ell_t$ recursively, as in the 
examples $\ell=4,5$ which have been described above.

\noindent{\bf Notation.} The following notation will be convenient. Write $\cR=\{0,1,2,\dots,\ell-2\}$,
$\ol\cR=\{1,2,3,\dots,\ell-1\}$ and $t\mapsto \ol t$ for the bijection $\cR\to\ol\cR$ given by $\ol t=\ell-1-t$.

We begin with two elementary observations concerning the $g$ function (see Theorem \ref{thm:tlcomp}).

\begin{lemma}\label{lem:g}
(1) If $i\in\N'$ and $i\equiv t\text {(mod }\ell)$, then 
\be\label{eq:g1}
g(i)=i+2\ol t.
\ee
(2) For any $i\in\N'=\{j\in\N\mid j\not\equiv -1\text{(mod }\ell)\}$, we have 
\be\label{eq:g2}
g^2(i)=i+2\ell.
\ee
\end{lemma}

These observations lead directly to the following formula for $\ell_t$.

\begin{corollary}\label{cor:lt}
For $t\in\cR$, we have
\be\label{eq:lt}
\ell_t=\sum_{i\in\N,\; i\equiv t\text {(mod }2\ell)}(w_i-w_{i+2\ol t}).
\ee
\end{corollary}
\begin{proof}
This is a direct consequence of Proposition \ref{prop:diml} and Lemma \ref{lem:g}. 
\end{proof}

We are now able to prove the key inductive relations.

\begin{proposition}\label{prop:red}
\begin{enumerate}
\item If $t\in\cR$ and $t\neq 0\text{ or }\ell-2$, then for $n\equiv t-1\text {(mod }2)$, we have
\be\label{eq:redgen}
\ell_t(n+1)=\ell_{t-1}(n)+\ell_{t+1}(n).
\ee
\item For $t=0$ and $n$ odd, we have
\be\label{eq:red0}
\ell_0(n+1)=\ell_1(n).
\ee
\item For $t=\ell-2$ and $n\equiv \ell-1\text {(mod }2)$, we have 
\be\label{eq:redl-2}
\ell_{\ell-2}(n+1)=\ell_{\ell-3}(n).
\ee
\end{enumerate}
\end{proposition}
\begin{proof}
We simply apply the relation \eqref{eq:recw2} to equation \eqref{eq:lt}. First observe that
if $i\equiv t\text{(mod }\ell)$, then 
\be\label{eq:all}
w_i(n+1)-w_{i+2\ol t}(n+1)=w_{i-1}(n)+w_{i+1}(n)-w_{i+2\ol t-1}(n)-w_{i+2\ol t+1}(n).
\ee
We shall combine the terms of the right side of \eqref{eq:all} in different ways, depending on the value of $t$.
First take $t$ such that $0<t<\ell-2$. Then, recalling that $\ol{t\pm 1}=\ol t\mp 1$, we have
\be\label{eq:gen}
\begin{aligned}
w_i(n+1)-w_{i+2\ol t}(n+1)=&(w_{i-1}(n)-w_{i+2\ol t+1}(n))+(w_{i+1}(n)-w_{i+2\ol t-1}(n))\\
=&(w_{i-1}(n)-w_{i-1+2\ol t+2}(n))+(w_{i+1}(n)-w_{i+1+2\ol t-2}(n))\\
=&(w_{i-1}(n)-w_{i-1+2\ol{ t-1}}(n))+(w_{i+1}(n)-w_{i+1+2\ol{ t+1}}(n)).\\
\end{aligned}
\ee

Now given \eqref{eq:lt},  summing both sides of \eqref{eq:gen} over $i\equiv t\text{(mod }\ell)$ yields the relation \eqref{eq:redgen}.

Next take $i\equiv 0\text{(mod }\ell)$, i.e. $t=0$. Then \eqref{eq:all} may be written as follows. For $i\equiv 0\text{(mod }\ell)$,
we have
\be\label{eq:0}
\begin{aligned}
w_i(n+1)-w_{i+2(\ell-1)}(n+1)=&(w_{i-1}(n)-w_{i-1+2\ell}(n))+(w_{i+1}(n)-w_{i+1+2(\ell-2)}(n))\\
=&(w_{i-1}(n)-w_{i-1+2\ell}(n))+(w_{i+1}(n)-w_{i+1+2(\ol{1})}(n)).\\
\end{aligned}
\ee

Summing both sides of \eqref{eq:0} over $i\equiv 0\text{(mod }\ell)$, we see that the first summand on the right is zero since 
all non-zero terms cancel, while the second summand is $\ell_1$ by \eqref{eq:lt}. This proves the relation \eqref{eq:red0}.

Finally, take $i\equiv \ell-2\text{(mod }\ell)$, i.e. $t=\ell-2$. In this case \eqref{eq:all} reads as follows.
\be\label{eq:l-2}
\begin{aligned}
w_i(n+1)-w_{i+2(\ol{\ell-2})}(n+1)=&w_{i-1}(n)+w_{i+1}(n)-(w_{i+1}(n)+w_{i+3}(n))\\
=&(w_{i-1}(n)-w_{i-1+2\ol{\ell-3}}(n)).\\
\end{aligned}
\ee

Summing both sides over $i\equiv \ell-2\text{(mod }\ell)$ yields the relation \eqref{eq:redl-2}
and completes the proof of the proposition.
\end{proof}

\begin{corollary}\label{cor:rec-l} We have the following recurrence for the dimensions $\ell_i(n)$.

\be\label{eq:recmat}
\begin{pmatrix}
\ell_0(n+1)\\
\ell_1(n+1)\\
\\
\\
.\\
.\\
.\\
\\
\\
\ell_{\ell-2}(n+1)\\
\end{pmatrix}
=
\begin{pmatrix}
0 & 1 & 0 & 0 & . & . & . &  &  & 0\\
1 & 0 & 1 & 0 & . & . & . &  &  & 0\\
0 & 1 & 0 &1  & 0 & . & . & . &  & 0\\
0 &0&1&0&1&.&.&.&&0\\
&&&.&.&.&&&&\\
&&&.&.&.&&&&\\
&&&.&.&.&&&&\\
0 & 0 & & . & . & . &  & 1 & 0 & 1\\
0 & 0 &  & . & . & . &  &0  &1  & 0\\
\end{pmatrix}
\begin{pmatrix}
\ell_0(n)\\
\ell_1(n)\\
\\
\\
.\\
.\\
.\\
\\
\\
\ell_{\ell-2}(n)
\end{pmatrix}
\ee
\end{corollary}

Note that since $\ell_i(n)\neq 0$ only if $i\equiv n\text{(mod }2)$, to usefully apply the relation \eqref{eq:recmat}, we need to square the 
$(\ell-1)\times(\ell-1)$ matrix in that relation. The result is conveniently formulated in terms of the following matrices. Define $m\times m$
matrices as follows.
{\small
\be\label{eq:defM}
\begin{aligned}
M_{11}(m)=\begin{pmatrix}
1 & 1 & 0 & 0 & . & . & . &  &  & 0\\
1 & 2 & 1 & 0 & . & . & . &  &  & 0\\
0 & 1 & 2 &1  & 0 & . & . & . &  & 0\\
0 &0&1&2&1&.&.&.&&0\\
&&&.&.&.&&&&\\
&&&.&.&.&&&&\\
&&&.&.&.&&&&\\
0 & 0 & & . & . & . &  & 1 & 2 & 1\\
0 & 0 &  & . & . & . &  &0  &1  & 1\\
\end{pmatrix}, &\;
M_{22}(m)=\begin{pmatrix}
2 & 1 & 0 & 0 & . & . & . &  &  & 0\\
1 & 2 & 1 & 0 & . & . & . &  &  & 0\\
0 & 1 & 2 &1  & 0 & . & . & . &  & 0\\
0 &0&1&2&1&.&.&.&&0\\
&&&.&.&.&&&&\\
&&&.&.&.&&&&\\
&&&.&.&.&&&&\\
0 & 0 & & . & . & . &  & 1 & 2 & 1\\
0 & 0 &  & . & . & . &  &0  &1  & 2\\
\end{pmatrix}\\
M_{12}(m)=\begin{pmatrix}
1 & 1 & 0 & 0 & . & . & . &  &  & 0\\
1 & 2 & 1 & 0 & . & . & . &  &  & 0\\
0 & 1 & 2 &1  & 0 & . & . & . &  & 0\\
0 &0&1&2&1&.&.&.&&0\\
&&&.&.&.&&&&\\
&&&.&.&.&&&&\\
&&&.&.&.&&&&\\
0 & 0 & & . & . & . &  & 1 & 2 & 1\\
0 & 0 &  & . & . & . &  &0  &1  & 2\\
\end{pmatrix}, &\;
M_{21}(m)=\begin{pmatrix}
2 & 1 & 0 & 0 & . & . & . &  &  & 0\\
1 & 2 & 1 & 0 & . & . & . &  &  & 0\\
0 & 1 & 2 &1  & 0 & . & . & . &  & 0\\
0 &0&1&2&1&.&.&.&&0\\
&&&.&.&.&&&&\\
&&&.&.&.&&&&\\
&&&.&.&.&&&&\\
0 & 0 & & . & . & . &  & 1 & 2 & 1\\
0 & 0 &  & . & . & . &  &0  &1  & 1\\
\end{pmatrix}\\
\end{aligned}
\ee
}

\begin{theorem}\label{thm:qdim-main}
Let $\ell=|q^2|\geq 4$. The dimensions $\ell_i(n)$ of the simple $Q_n(\ell)$-modules are determined by the initial conditions
$\ell_i(\ell-1)=w_i(\ell-1)$ for $i\equiv \ell-1\text{(mod }2)$, $\ell_i(\ell-2)=w_i(\ell-2)$ for $i\equiv \ell\text{(mod }2)$, as well as the following
recurrences. Let $\ell\geq 6$. If $\ell$ is even, then
{\small
\be\label{eq:rec11}
\bl_{00}(n+2):=
\begin{pmatrix}
\ell_0(n+2)\\
\ell_2(n+2)\\
\\
\\
.\\
.\\
.\\
\\
\\
\ell_{{\ell-2}}(n+2)\\
\end{pmatrix}
=
M_{11}\left(\frac{\ell}{2}\right)
\begin{pmatrix}
\ell_0(n)\\
\ell_2(n)\\
\\
\\
.\\
.\\
.\\
\\
\\
\ell_{\ell-2}(n)\\
\end{pmatrix}
\ee
and 
\be\label{eq:rec22}
\bl_{01}(n+2):=
\begin{pmatrix}
\ell_1(n+2)\\
\ell_3(n+2)\\
\\
\\
.\\
.\\
.\\
\\
\\
\ell_{\ell-3}(n+2)\\
\end{pmatrix}
=
M_{22}\left(\frac{\ell}{2}-1\right)
\begin{pmatrix}
\ell_1(n)\\
\ell_3(n)\\
\\
\\
.\\
.\\
.\\
\\
\\
\ell_{\ell-3}(n)\\
\end{pmatrix}.
\ee

If $\ell\geq 6$ is odd, then
\be\label{eq:rec12}
\bl_{10}(n+2):=
\begin{pmatrix}
\ell_0(n+2)\\
\ell_2(n+2)\\
\\
\\
.\\
.\\
.\\
\\
\\
\ell_{\ell-3}(n+2)\\
\end{pmatrix}
=
M_{12}\left(\frac{\ell-1}{2}\right)
\begin{pmatrix}
\ell_0(n)\\
\ell_2(n)\\
\\
\\
.\\
.\\
.\\
\\
\\
\ell_{\ell-3}(n)\\
\end{pmatrix}
\ee
and 
\be\label{eq:rec21}
\bl_{11}(n+2):=
\begin{pmatrix}
\ell_1(n+2)\\
\ell_3(n+2)\\
\\
\\
.\\
.\\
.\\
\\
\\
\ell_{\ell-2}(n+2)\\
\end{pmatrix}
=
M_{21}\left(\frac{\ell-1}{2}\right)
\begin{pmatrix}
\ell_1(n)\\
\ell_3(n)\\
\\
\\
.\\
.\\
.\\
\\
\\
\ell_{\ell-2}(n)\\
\end{pmatrix}.
\ee
}
If $\ell=4$ or $5$ then as noted in the exposition above, the respective recurrence matrices are $\begin{pmatrix} 1&1\\1&1\\ \end{pmatrix}$, 
$\begin{pmatrix}  2 \end{pmatrix}$ and $\begin{pmatrix} 1&1\\2&1\\ \end{pmatrix}$, $\begin{pmatrix} 2&1\\1&1\\ \end{pmatrix}$.
\end{theorem}

\begin{remark}\label{rem:mats}
We remark that our matrices $M_{22}$ and $M_{12}$ coincide with the matrices occurring in \cite[\S 4.2]{R}. However the context appears 
to be different.
\end{remark}

To state Theorem \ref{thm:qdim-main} more explicitly, and in order to give explicit formulae for the dimension
of the algebras $Q_n(\ell)$, we introduce the following notation. For $\ell\geq 6$, 
define column vectors $\bw_{ij}$ for $i,j\in\{0,1\}$
as follows.

\be\label{eq:wijev}
\text{For $\ell$ even, define }
\bw_{00}=\begin{pmatrix}
w_0(\ell-2)\\
w_2(\ell-2)\\
\\
\\
.\\
.\\
.\\
\\
\\
w_{\ell-2}(\ell-2)\\
\end{pmatrix}
\text{ and }
\bw_{01}=\begin{pmatrix}
w_1(\ell-3)\\
w_3(\ell-3)\\
\\
\\
.\\
.\\
.\\
\\
\\
w_{\ell-3}(\ell-3)\\
\end{pmatrix}.
\ee
\be\label{eq:wijod}
\text{For $\ell$ odd, define }
\bw_{10}=\begin{pmatrix}
w_0(\ell-3)\\
w_2(\ell-3)\\
\\
\\
.\\
.\\
.\\
\\
\\
w_{\ell-3}(\ell-3)\\
\end{pmatrix}
\text{ and }
\bw_{11}=\begin{pmatrix}
w_1(\ell-2)\\
w_3(\ell-2)\\
\\
\\
.\\
.\\
.\\
\\
\\
w_{\ell-2}(\ell-2)\\
\end{pmatrix}.
\ee

Theorem \ref{thm:qdim-main} may now be stated as follows.
\begin{corollary}\label{cor:main}
If $\ell$ is even, then 
\be\label{eq:main1}
\bl_{00}(2n)=M_{11}^{n-\frac{\ell}{2}+1}\bw_{00}, and
\ee
\be\label{eq:main2}
\bl_{01}(2n+1)=M_{22}^{n-\frac{\ell}{2}+2}\bw_{01}.
\ee

If $\ell$ is odd, then
\be\label{eq:main1}
\bl_{10}(2n)=M_{12}^{n-\frac{\ell-3}{2}}\bw_{10}, and
\ee
\be\label{eq:main1}
\bl_{11}(2n+1)=M_{21}^{n-\frac{\ell-3}{2}}\bw_{11}.
\ee
\end{corollary}

\subsection{The dimension of $Q_n(\ell)$} Using the fact that the dimension of a semisimple $\C$-algebra is equal to the sum of the
squares of the dimensions of its simple modules, we may now write the following closed formulae for the dimension of $Q_n(\ell)$,
bearing in mind that the matrices $M_{ij}$ are all symmetric..

\begin{theorem}\label{thm:dimq}
If $\ell\geq 6$ is even, then 
$$\dim(Q_{2n}(\ell))=\bw_{00}^t M_{11}^{2n-\ell+2}\bw_{00},$$ and
$$\dim(Q_{2n+1}(\ell))=\bw_{01}^t M_{22}^{2n-\ell+4}\bw_{01}.$$

If $\ell$ is odd, then $$\dim(Q_{2n}(\ell))=\bw_{10}^t M_{12}^{2n-\ell+3}\bw_{10},$$ and
$$\dim(Q_{2n+1}(\ell))=\bw_{11}^t M_{21}^{2n-\ell+3}\bw_{11}.$$
\end{theorem}

\subsection{The case $\ell=6$} We give the explicit solutions for the case $\ell=6$.
\begin{theorem}\label{thm:l=6}
Let $\ell=6$. Then for $n\geq 1$,
\begin{enumerate}
\item For $n\geq 1$, we have $\ell_1(2n+1)=\dim(L_1(2n+1))=\frac{3^n+1}{2}$ and $\ell_3(2n+1)=\dim(L_3(2n+1))=\frac{3^n-1}{2}$.
\item $\dim(Q_{2n+1}(6))=\frac{3^{2n}+1}{2}$.
\item For $n\geq 1$, we have $\ell_0(2n)=\dim(L_0(2n))=\frac{3^{n-1}+1}{2}$, $\ell_2(2n)=\dim(L_2(2n))=3^{n-1}$
 and $\ell_4(2n)=\dim(L_4(2n))=\frac{3^{n-1}-1}{2}$.
\item For $n\geq 1$, we have $\dim(Q_{2n}(6))=\frac{3^{2n-1}+1}{2}$.
\end{enumerate}
\end{theorem}

Note that the statements (2) and (4) of the theorem may be combined to state that for $n\geq 2$, 
\be\label{eq:dimq-6}
\dim(Q_n(6))=\frac{3^{n-1}+1}{2}.
\ee
\begin{proof}
It follows from Theorem \ref{thm:qdim-main} that the dimensions $\ell_1(3+2n)$ and $\ell_3(3+2n)$ of the relevant simple modules
are given by the equation 
\be\label{eq:6-odd}
\begin{pmatrix}
\ell_1(3+2n)\\ \ell_3(3+2n)\\
\end{pmatrix}=
\begin{pmatrix}
2 & 1\\ 1 & 2\\
\end{pmatrix}^n
\begin{pmatrix}
2\\1\\
\end{pmatrix}
\ee 
But an easy calculation shows that 
\be
\begin{pmatrix}
2 & 1\\ 1 & 2\\
\end{pmatrix}^n=
\begin{pmatrix}
\frac{3^n+1}{2} & \frac{3^n-1}{2}\\ \frac{3^n-1}{2} & \frac{3^n+1}{2}\\
\end{pmatrix},
\ee
and the statement (1) follows.

The dimension of $Q_{2n+1}(6)$ is the sum of the squares of the dimensions of its simple modules, so that (2) follows immediately from (1).

Similarly, it also follows from Theorem \ref{thm:qdim-main} that the dimensions $\ell_i(2n+2)$ of the simple modules $L_i(2n+2)$ for $i$ even,
are given by the equation 
\be\label{eq:6-even}
\begin{pmatrix}
\ell_0(2n+2)\\ \ell_2(2n+2)\\\ell_4(2n+2)
\end{pmatrix}=
\begin{pmatrix}
1 & 1 & 0\\ 1 & 2 & 1\\0& 1 & 1\\
\end{pmatrix}^n
\begin{pmatrix}
1\\1\\0
\end{pmatrix}
\ee 

Further, another easy calculation shows that for $n\geq 1$,
\be
\begin{pmatrix}
1 & 1 & 0\\ 1 & 2 & 1\\ 0 & 1 & 1\\
\end{pmatrix}^n=
\begin{pmatrix}
\frac{3^{n-1}+1}{2} & 3^{n-1} & \frac{3^{n-1}-1}{2}\\ 3^{n-1} & 2\times 3^{n-1} & 3^{n-1}\\
 \frac{3^{n-1}-1}{2} & 3^{n-1} & \frac{3^{n-1}+1}{2}\\
\end{pmatrix},
\ee

and substituting into \eqref{eq:6-even}, we find that for any positive integer $n\geq 0$,

\be\label{eq:6-even}
\begin{pmatrix}
\ell_0(2n+2)\\ \ell_2(2n+2)\\\ell_4(2n+2)
\end{pmatrix}=
\begin{pmatrix}
\frac{3^{n}+1}{2}\\ 3^n\\ \frac{3^{n}-1}{2}\\
\end{pmatrix}.
\ee

The statement (3) follows immediately, and (4) follows by taking the sum of the squares of the dimensions
of the three simple modules for $Q_n(6)$.
\end{proof}

\subsection{Some speculation}\label{ss:spec} We conclude with some speculations on possible connections of our results with 
Virasoro algebras. Recall that the Virasoro algebra $\cL=\oplus_{i\in\Z}\C L_i\oplus \C C$ has irreducible highest weight modules
$L(c,h)$ {with highest weight $(c,h)$}, where $c,h(\in\C)$ are respectively the central charge and the eigenvalue of $L_0$. 
%{\color{blue}on the highest weight subspace of}  $L(c,h)$\footnote{\color{blue}we can also remove  \textit{on the highest weight subspace of $L(c,h)$.}}. 
It was conjectured by Friedan, 
Qiu and Schenker \cite{FQS1} and proved by Langlands \cite{L} that $L(c,h)$ is unitarisable if and only if either
\begin{enumerate}
\item $c\geq 1$ and $h\geq 0$, or
\item there exist integers $m\geq 2${, $r$} and $s$ with $1\leq s\leq r<m$ such that
\[
c=1-\frac{6}{m(m+1)} \text{  and  }h=\frac{\left((m+1)r-ms\right)^2-1}{4m(m+1)}.
\]
\end{enumerate}
This result bears a superficial resemblance to Jones' result on the range of values of the index of a subfactor
%which 
{as} was mentioned in the preamble. Thus it might be expected that case (2) is somehow connected
with our algebras $Q_n(\ell)$ for $\ell=3,4,5,\dots$.

Further, there are several instances in the literature (see, e.g. {\cite{GS, KS, N}}) which hint at a connection between $Q_n(\ell)$
and the minimal unitary series of $\cL$ with central charge $c=1-\frac{6}{\ell(\ell-1)}$. Our work may provide some further
evidence along those lines.

For $\ell=3$, $c=0$, and there is just one irreducible representation, viz. the trivial one. This is `consistent' with $Q_n(3)={\C}$.
For $\ell=4$, $c=\frac{1}{2}$. This case is the Ising model, or equivalently, the $2$-state Potts model, as we have already observed. 

For $\ell=5$, which gives rise to the sequence of ``Fibonacci algebras'' $Q_n(5)$,
$c=\frac{7}{10}$, leading to interesting combinatorics involving the Rogers-Ramanujan identities. This is the 
``tricritical Ising model'', {which has $N=1$ supersymmetry}. 
%where there are $6$ unitary representations. 
For $\ell=6$, $c=\frac{4}{5}$, and we have the  
$3$-state Potts model {with $\Z_3$-parafermionic symmetry}.

We hope to return to this theme in a future work.
%%%%%%%%%%%%

\end{document}